\documentclass[10pt,a4paper]{amsart}
\usepackage{amssymb}
\usepackage{amsmath}
\usepackage{amsfonts,a4wide}
\usepackage{bbm}
\usepackage{enumerate}
\usepackage[latin1]{inputenc}
\usepackage{shadow}
\usepackage{slashed}
\usepackage[all]{xy}
\usepackage{tikz-cd}
\usepackage[letterpaper, margin=3cm]{geometry}

\theoremstyle{definition}
\newtheorem{thm}[subsection]{Theorem}

\newtheorem{prop}[subsection]{Proposition}

\newtheorem{cor}[subsection]{Corollary}
\newtheorem{lemma}[subsection]{Lemma}

\newtheorem{remark}[subsection]{Remark}
\newtheorem{example}[subsection]{Example}

\newcommand{\J}{\mathcal{J}}

\newcommand{\R}{\mathbb{R}}
\newcommand{\RP}{\mathbb{RP}}
\newcommand{\CP}{\mathbb {CP}}

\newcommand{\Z}{\mathbb{Z}}
\newcommand{\ZZ}{\mathbb{Z}}

\newcommand{\JJ}{\mathcal J}
\newcommand{\jj}{\texttt{j}}
\newcommand{\C}{\mathbb{C}}
\newcommand{\OO}{\mathbb{O}}
\newcommand{\HH}{\mathbb{H}}

\DeclareMathOperator{\Hom}{Hom}

\DeclareMathOperator{\ev}{ev}

\DeclareMathOperator{\Gr}{Gr}

\DeclareMathOperator{\Map}{Map}

\DeclareMathOperator{\Spin}{Spin}

\begin{document}

\title{On the topology of the space of almost complex structures on the six sphere}

\author{Bora Ferlengez}
\address{Deparment of Mathematics and Statistics, Hunter College of CUNY, 695 Park Avenue, 01220 New York}
\email{bora.ferlengez@gmail.com}

\author{Gustavo Granja}
\address{Center for Mathematical Analysis, Geometry and Dynamical Systems, Instituto Superior T\'ecnico,
Universidade de Lisboa, Av. Rovisco Pais, 1049-001 Lisboa}
\email{gustavo.granja@tecnico.ulisboa.pt}

\author{Aleksandar Milivojevi\'c}
\address{Max Planck Institute for Mathematics,
Vivatsgasse 7,
53111 Bonn}
\email{milivojevic@mpim-bonn.mpg.de}

\begin{abstract} The space of orientation-compatible almost complex structures on $S^6$ naturally contains a copy of $\RP^7$. We show that the inclusion induces an isomorphism on fundamental groups and rational homotopy groups. We also compute the homotopy fiber of the inclusion and the homotopy groups of the space of almost complex structures in terms of the homotopy groups of $S^7$. Our approach lends itself to generalization to components of almost complex structures with $c_1 = 0$ on six--manifolds. \end{abstract}

\maketitle

\section{Introduction}

In this note we consider the topology of the space of almost complex structures $\mathcal{J}(S^6)$ on the six--dimensional sphere $S^6$. By an almost complex structure we will mean one which induces a fixed orientation and which is orthogonal with respect to some fixed metric on $S^6$; swapping the orientation, varying the metric, or removing the metric compatibility condition altogether results in a homotopy equivalent space via the appropriate natural map. For this reason we from now on take $S^6$ to be the unit sphere in the imaginary octonions endowed with the round metric and the standard orientation.

We will employ the useful description of $\JJ(S^6)$ as the space of sections of the natural $SO(6)/U(3) \cong \CP^3$ bundle over $S^6$ whose fibers are the complex structures on the given tangent vector space. Likely one would be interested primarily in the smooth sections of this bundle, corresponding to smooth almost complex structures; the inclusion of the space of smooth sections into the space of continuous sections is a homotopy equivalence, so we will study the latter without loss of generality. 

Let $J^{\textrm{cn}}$ denote the canonical almost complex structure given on $T_pS^6$ by octonion multiplication 
$$
J^{\textrm{cn}}_p(v)=pv.
$$


The group $SO(7)$ acts naturally on
 the space of sections by
$$
(A \cdot J)_p(v)= AJ_{A^{-1}p}(A^{-1}v).
$$

For our canonical almost complex structure $J^{\textrm{cn}}$ we have that $A\cdot J^{\textrm{cn}}=J^{\textrm{cn}}$ means
$$ A(pv) =(Ap)(Av), $$
and so the isotropy of this action is the group $G_2$ of real algebra automorphisms of the octonions. We thus obtain an inclusion
$$ SO(7)/G_2 \cong \RP^7 \stackrel{j}{\hookrightarrow} \JJ(S^6). $$

This subspace was previously studied by Battaglia \cite{Ba}. Her description of the space of complex structures on $\R^6$ in terms of octonion multiplication is crucial to the present work (see Proposition \ref{csr6}). 

The aim of this note is to show that the homotopy fiber of the inclusion $\RP^7 \stackrel{j}{\hookrightarrow} \JJ(S^6)$ is a path component of the sevenfold based loop space $\Omega^7S^7$ of the seven--sphere. As the homotopy groups of the latter space are torsion, we recover the result of the first named author \cite[\textsection 4]{F18} that $\JJ(S^6)$ has the rational homotopy groups of $S^7$ (or $\RP^7$), and see that the natural inclusion of $\RP^7$ induces an isomorphism on rational homotopy groups and on fundamental groups. The non-contractibility of the homotopy fiber of the inclusion also answers a question communicated by Dennis Sullivan and stated in loc. cit., on whether this inclusion is a genuine (i.e. integral) homotopy equivalence.

We then touch upon the general problem of the topology of the space of almost complex structures on a closed six--manifold in the 
component of an almost complex structure with $c_1 = 0$. This, along with the case of $c_1 \neq 0$, will be studied more systematically in upcoming work.

\subsection*{Notation and conventions} Given spaces $X$ and $Y$, $\Map(X,Y)$
denotes the space of (continuous) maps with the compact-open topology. 
Spaces of sections, and spaces of pointed maps are topologized as
subspaces of the relevant mapping spaces. If $X$ and $Y$ are pointed spaces, we write $\Omega^k X$ for the $k$-fold loop space on $X$, i.e. the space $\Map_*(S^k, X)$ of pointed maps from $S^k$ to $X$, and by $[X,Y]_*$ we denote pointed homotopy classes of maps. 

\subsection*{Acknowledgements} We thank Luis Fernandez and Scott Wilson for useful comments, and the anonymous referee for helping improve the exposition; the third author likewise thanks Claude LeBrun and Dennis Sullivan. The second author was partially supported by FCT/Portugal through project UIDB/MAT/04459/2020. 

\section{Complex structures on $\R^6$}
We use \cite{CS} as a reference for the normed division algebra $\OO$ of the octonions. This algebra is obtained from the 
quaternions $\HH$ by the Cayley--Dickson construction: given quaternions $a,b$, an octonion is a pair of quaternions $(a,b)$, written $a+Ib$, with 
 multiplication rule given by 
\begin{equation}
\label{CD}
(a+Ib)(c+Id)=(ac-d\overline b) + I(cb+\overline a d)
\end{equation}
The previous formula can be applied more generally. A composition algebra is 
a normed, unital algebra over $\R$ where the norm comes from an inner product on the underlying vector space. Such an algebra
has a conjugation map defined by $\overline x = 2 \langle x , 1 \rangle 1 - x$.
Given a composition algebra $C$, let $A$ be a subalgebra (containing $1$) and 
suppose $I$ is a unit vector in $C$ orthogonal to $A$. Then one can show 
(see \cite[Sections 6.1-6.3]{CS}) that $IA$ is orthogonal to $A$ and $A \oplus IA$ 
is a subalgebra of $C$ with multiplication given by \eqref{CD}.

In particular, if we pick any unit imaginary octonion $x$, we obtain a copy $\langle 1, x\rangle$ of 
the complex numbers in $\OO$. If we further pick a unit vector $y \in \langle 1,x \rangle^\perp$
we obtain a copy of the quaternions $\langle 1,x,y,yx\rangle$ which we call a quaternion subalgebra $A \subset \OO$. A choice of 
a unit vector orthogonal to $A$ then allows us to express the 
multiplication in $\OO$ in terms of pairs of elements of $A$ by
the formula \eqref{CD}.
 
Note that, as a consequence of the above discussion, the
group $G_2$ of automorphisms of $\OO$ acts transitively on unit vectors orthogonal 
to $\R$ (and, in fact, simply transitively on the orthonormal $3$--frames $(x,y,z)$ in $\R^\perp$ such 
that $z \perp yx$).
 
We identify $\R^6$ with the orthogonal complement to $\C = \langle 1,i \rangle$,
$$
\R^6 = \langle 1,i \rangle^\perp \subset \OO.
$$
A complex structure $J$ on $\R^6$ is said to be compatible with the standard orientation if
a complex basis $(u,v,w)$ for $\R^6$ gives rise to a real basis $(u,Ju,v,Jv,w,Jw)$ with the standard orientation. Given 
a unit octonion $x$, let $R_x \in SO(8)$ denote right multiplication by the octonion $x$.

\begin{prop}
\label{csr6}
Let $J$ be an orthogonal complex structure on $\R^6$ compatible with the standard orientation. The following (equivalent) statements hold:

\begin{enumerate} \item $J$ is obtained by conjugating 
the standard complex structure (given by left multiplication by $i$) by the map $R_{\overline x}$ for some unit octonion $x$, which is unique up to left multiplication by an element of $S^1 \subset \langle 1, i \rangle = \C$. Namely, $J$ is of the form
$$
J_x(v) = (i(vx))\overline x  =: i^{R_{\overline x}}(v).
$$
\item $J$ has a complex line in common with $i$, i.e. there is a complex line $L$ on $\R^6$ (with complex structure given by $i$) such that $J$ restricted to $L$ is equal to (left multiplication by) $i$. \end{enumerate}
\end{prop}
\begin{proof} 
The formula above for $J_x$ can be understood as follows using quaternion coordinates: let $A$ denote a quaternion subalgebra 
of $\OO$ containing $1,i$ and $x$ (which is unique unless $x\in \C$). Fix a standard basis $\langle 1, i, \jj, i\jj \rangle$ for
this algebra and write
$$
x = \cos \theta + w \sin \theta 
$$
with $w$ a unit vector in the $\langle i, \jj, i\jj \rangle$ plane. Let $I \in \OO$ be a unit vector orthogonal to $A$ and write
an element $v \in \R^6$ as 
$$
v = a + I b
$$
with\footnote{ We are interested in the case $a=z\jj$ with $z \in \C$ but the computation holds for an arbitrary $a \in A$.}
$a,b \in A$. Using \eqref{CD} we see that 
\begin{eqnarray*}
(i(v x))\overline x & = &  (i((a+Ib)(\cos \theta + w \sin \theta)) )(\cos \theta - w \sin \theta) \\
& = & ( ia(\cos \theta + w\sin \theta) + I((-i)((\cos \theta + w \sin\theta)b)   ))(\cos \theta - w \sin \theta) \\
& = & ia +I ((\cos \theta - w \sin \theta)(-i)(\cos \theta + w \sin \theta) b) \\
& = & ia + ((\cos \theta - w\sin \theta)i(\cos \theta + w \sin \theta))(Ib)
\end{eqnarray*}
This means that $J_x$ acts as
\begin{itemize}
\item left multiplication by $i$ on the quaternion subalgebra $A$,
\item left multiplication by $l$ on the orthogonal plane $A^\perp$ where $l$ is the unit quaternion
obtained from $i$ by rotating the unit imaginary quaternions in $A$ by an angle of $2\theta$ around the axis $w$.
\end{itemize}

If $x$ and $y$ are unit octonions such that $J_x=J_y$ then $x$ and $y$ must both be
contained in the same quaternion 
subalgebra $A=\langle 1,i,\jj,i\jj \rangle \subset \OO$. If we write $x=z_1 + z_2 \jj$ and $y=w_1+w_2 \jj$ with $z_i,w_i \in \C$, the equality of $J_x$ and $J_y$
on $A^\perp$ translates to
$$
(\overline{z_1} - z_2 \jj )i(z_1 + z_2 \jj ) = (\overline{w_1} - w_2 \jj )i(w_1+w_2 \jj),
$$
and one easily checks this amounts to the existence of $e^{i\theta}$ such that $y=e^{i\theta}x$.

We first prove the equivalence of statements (1) and (2). That (1) implies (2) is immediate from the description of $J_x$ in quaternion coordinates.
Conversely, assume that $J$ is an almost complex structure agreeing with $i$ on the complex line $L\subset \R^6$. Then
$A=\C\oplus L$ is a quaternion subalgebra of $\OO$. Since $J$ is orthogonal and preserves orientation, its restriction
to $A^\perp$ must be given by left multiplication by a unit imaginary element $l \in A$. Clearly there exists 
$x= \cos \theta + w \sin\theta \in A$ such that $\overline x i x = l$ and, for such an $x$, we have $J=J_x$. 
We have already seen that $x$ is unique up to left multiplication by a unit complex number. 

Statement (1) follows immediately from \cite[Proposition 3.5]{Ba}, but for completeness we give a proof of (2). 
A complex structure $J$ on $\R^6$ determines a subspace
$$\CP^2_J \subset \Gr^+_2(\R^6)$$ 
consisting of $J$-complex lines inside $\Gr^+_2(\R^6)$ (the Grassmannian of oriented planes). 
Since the space of complex structures compatible with the orientation is path connected, the fundamental 
class $[\CP^2_J] \in H_4(\Gr^+_2(\R^6))$ is independent of $J$. The self intersection of $\CP^2$ in
$\Gr_2^+(\R^6)$ is given by the Euler number of the normal bundle, which we check to be $1$ in Lemma \ref{eulerclass} below. It follows that, 
for any two complex structures $J,J'$, the subspaces $\CP^2_J$ and $\CP^2_{J'}$ intersect, i.e. $J$ and $J'$
have a complex line in common. 
If $J$ and $J'$ are orthogonal and distinct, this line is necessarily unique, and $J$ and $J'$ must agree on that line.
\end{proof}

\begin{lemma}\label{eulerclass} 
The Euler number of the normal bundle to $\CP^2$ in $\Gr_2^+(\R^6)$ is 1.
\end{lemma}
\begin{proof}
Near each $P \in \Gr_2^+(\R^6)$ we have standard coordinates with values in $\Hom(P,P^\perp) \cong \Hom(\R^2,\R^4)$. If $P \in \CP^2$, then
both $P$ and $P^\perp$ are complex planes and hence $\Hom(P,P^\perp)$ breaks up into a direct sum of the subspaces of 
complex linear and complex anti-linear maps
$$
\Hom(P,P^\perp) = \Hom_\C(P,P^\perp) \oplus \Hom_{\overline \C}(P,P^\perp).
$$
The first summand can naturally be identified with the tangent space $T_P \CP^2$, while the second can be naturally identified with the fiber over $P$ of the normal bundle $N\CP^2$. 
Our aim is to compute $c_2(N\CP^2)$, as the Euler number is
obtained by evaluating this cohomology class on the fundamental
class of $\CP^2$.

Let $\tau$ denote the tautological bundle over $\CP^2$, and 
$\tau^\perp$ its orthogonal complement in the trivial bundle $\CP^2 \times \C^3$.
Since $\Hom_{\overline \C}(P,P^\perp) = (\overline P)^* \otimes_\C P^\perp \cong P \otimes_\C P^{\perp}$,
we have $N\CP^2 \cong \tau \otimes \tau^\perp$.

Let $a=c_1(\tau) \in H^2(\CP^2)$. Then $(1+a)(1+c_1(\tau^\perp)+c_2(\tau^\perp))=1$ and therefore 
$c_1(\tau^\perp)=-a$, $c_2(\tau^\perp)=a^2$. This computation gives us the homomorphism
induced on cohomology by the map $\CP^2 \to BU(1)\times BU(2)$ which expresses $N\CP^2 \cong \tau \otimes \tau^\perp$ as the pullback of the tensor product $\tau_1 \otimes \tau_2$ of the tautological 
bundles  over $BU(1)\times BU(2)$. It remains to compute $c_2(\tau_1\otimes \tau_2)$.

The restriction of $\tau_1 \otimes \tau_2$ to the classifying space of the standard maximal torus is the bundle 
\begin{equation}
\label{bundlemaxtor}
EU(1)^3 \otimes_{ U(1)^3} \C^2 \to BU(1)^3 
\end{equation}
where $U(1)^3$ acts on $\C^2$ via the character 
\begin{equation}
\label{char}
\chi(e^{i\theta}, e^{i\alpha},e^{i\beta}) = (e^{i(\alpha+\theta)}, e^{i(\beta+\theta)})
\end{equation}
Let $x_i$ for $i=1,2,3$ denote the Chern class of the tautological line bundle over each of the $BU(1)$
factors. The second Chern class of the bundle \eqref{bundlemaxtor} is 
$(x_1+x_2)(x_1+x_3)= x_1^2 + x_1(x_2+x_3) + x_2 x_3$.
Therefore $c_2( \tau_1 \otimes \tau_2) = c_1(\tau_1)^2 +c_1(\tau_1) c_1(\tau_2) + c_2(\tau_2)$ and hence
$$
c_2( N\CP^2) = a^2 + a(-a) + a^2 = a^2
$$
which completes the proof. 
\end{proof}

\begin{remark}
\label{remjr6}
Proposition \ref{csr6} gives rise to the following geometric description of the space $J(\R^6)$ of orthogonal complex 
structures on the vector space $\R^6$ compatible with the standard orientation: Let $\CP^2$ be the space of complex lines in $\R^6$ 
for the complex structure $i$.
The tautological line bundle $\tau \to \CP^2$ gives rise to a bundle of quaternion algebras 
$\tau\oplus \langle 1,i \rangle \to \CP^2$. There is a map from the unit sphere bundle 
$S(\tau \oplus \C) \to J(\R^6)$ sending $x \mapsto i^{R_x}$, which descends to 
the quotient by the natural action of $S^1 \subset \C$.  This yields
a surjective map $\phi \colon P(\tau\oplus \C) \to J(\R^6)$ which collapses 
$P(0\oplus \C) \subset P(\tau\oplus \C)$ to the point $i$ and 
is bijective on the complement of $P(0\oplus \C)$. The map $\phi$ expresses $J(\R^6)$ as a
blow-down of $\CP^3 \sharp \overline{\CP^3}$.

Alternatively, a complex structure other than $i$ determines an element of $\CP^2$ (the line on which it coincides with $i$). 
Identifying $J(\R^6)$ with $\CP^3$, we obtain a map
$$ \CP^3 \setminus \{i\}  \to \CP^2 $$
with fiber the affine line $\CP^1 \setminus \ast$ of all possible orthogonal complex structures on the orthogonal $\R^4$ other than 
$i$ itself. Considering the section given by assigning to $L \in \CP^2$ the complex structure given by $i$ on $L$ and $-i$ on 
$L^\perp$, we may regard this fibration as the standard fibration of $\CP^3\setminus \ast$ by linear projection onto  
$\CP^2$ from a point at $\infty$ on a complementary line. 
\end{remark}

\section{Almost complex structures on $S^6$}

Recall that $S^6$ denotes the unit sphere in the imaginary octonions equipped with the 
round metric and the standard orientation. 
Each $p \in S^6$ determines (by left octonion multiplication) an orthogonal complex structure
on $\R^8$ and on $T_p S^6 = \langle 1, p \rangle^\perp \subset \OO$. 
Since there is an automorphism of the octonions taking $i$ to $p$, it follows from Proposition \ref{csr6} that an arbitrary orthogonal complex structure on
$T_{p}S^6$ can be written as $p^{R_{\overline{x}}}$ for some unit octonion $x$.

For an oriented Riemannian $2n$--manifold $M$ we refer to the natural $SO(2n)/U(n)$ bundle over $M$, whose fibers are the complex structures on the given tangent vector space, as the \textit{twistor space} $Z(M)$. We have the following description of the twistor space $Z(S^6) \cong SO(7)/U(3)$:

\begin{prop} 
\label{funnyaction}
Consider the $S^1$-action on $S^6 \times S^7$ defined by
$$ e^{i\theta} \cdot(p, x) = (p, (\cos \theta + p \sin \theta) x) $$
The map 
$$ (S^6 \times S^7)/S^1 \xrightarrow{\psi} Z(S^6) $$
given by
$$ [(p, x)] \mapsto  p^{R_{\overline x}}  $$
is  a diffeomorphism.
\end{prop}
\begin{proof}
The quaternion subalgebras containing $p$ and $x$ are the same as the quaternion subalgebras
containing $p$ and $(\cos \theta + p \sin \theta)x$. The computation in the proof of 
Proposition \ref{csr6} shows that 
$$ (p(v( (\cos \theta + p \sin \theta) x)))\overline{((\cos \theta + p \sin \theta) x)} =
(p (v x))\overline x $$
so that $\psi$ is well defined. Again by Proposition \ref{csr6}, the map $\psi$ is a fiberwise bijection and
hence a fiberwise homeomorphism. Clearly $\psi$ is a smooth map. We leave it to the reader to check that 
$\psi^{-1}$ is also smooth.
\end{proof}

\begin{remark}
Recall that $G_2$ acts transitively on $S^6$ and
the isotropy group of $i$ can be identified with $SU(3)$. The fibration 
$$ SO(6)/U(3) \to SO(7)/U(3) \to S^6 $$
is associated to the principal fibration
\begin{equation}
\label{prin} SU(3) \to G_2 \to S^6
\end{equation}
where $SU(3)$ acts on $\CP^3 = SO(6)/U(3)$ fixing a distinguished point (in the fiber over $p$, this is the tautological
complex structure $p$). Indeed, we have a map 
$$ G_2 \times S^7 \xrightarrow{\Psi} S^6 \times S^7 $$
given by 
$$ (\phi,x) \mapsto ( \phi(i), \phi(x)) $$
If $\lambda \in SU(3)$, we have $\Psi(\phi\lambda, x) = \Psi(\phi, \lambda(x))$ so 
$\Psi$ descends to a diffeomorphism $\overline \Psi \colon G_2 \times_{SU(3)} S^7 \to S^6 \times S^7$.
As $\Psi$ conjugates the standard $S^1$ action on the
second coordinate of $G_2 \times S^7$ to the action of 
Proposition \ref{funnyaction}, we see that $\overline \Psi$ induces a diffeomorphism
$$
G_2 \times_{SU(3)} \CP^3 \to (S^6 \times S^7)/S^1.
$$
Note that the clutching map for \eqref{prin} is a generator of 
$\pi_5 SU(3)$ (see for instance \cite{G}) 
which becomes trivial in $\pi_5 G_2$, so the bundle of normed division algebras associated to 
\eqref{prin} is trivial. The map $\overline \Psi$ is an explicit trivialization.
\end{remark}

Consider the map of fibrations over $S^6$
$$
\xymatrix{
S^6 \times S^7 \ar[r] \ar[d] & Z(S^6) \ar[d] \\
S^6 \ar[r] & S^6
}
$$

where the top map is obtained by taking the quotient by the $S^1$-action and then applying $\psi$ from Proposition \ref{funnyaction}. Taking sections and evaluating at $i$ gives rise to a square diagram of fibrations
$$
\xymatrix{
\Map(S^6,S^7) \ar[r]^{\varpi} \ar[d]^{\ev_i} &  \JJ(S^6) \ar[d]^{\ev_i} \\
S^7 \ar[r] & \CP^3
}
$$
where we have identified $J(\R^6)$ with $\CP^3$ (using Proposition \ref{csr6}).
This diagram can be completed to a $3\times 3$ diagram of fiber sequences 
\begin{equation}
\label{maind}
\xymatrix{
\Map_*(S^6,S^1) \ar[r] \ar[d] & \Omega^6 S^7 \ar[r]^{\varpi'} \ar[d] & \JJ_*(S^6) \ar[d] \\
\Map(S^6,S^1) \ar[r]^\iota \ar[d]^{\ev_i} & \Map(S^6,S^7) \ar[r]^{\varpi} \ar[d]^{\ev_i} &  \JJ(S^6) \ar[d]^{\ev_i} \\
S^1 \ar[r] & S^7 \ar[r] & \CP^3
}
\end{equation}
where $\JJ_*(S^6)=\{ J \in \JJ(S^6) \colon J(i)=i\}$, and $\Map_*(S^6,S^1)$ is the (contractible) space
of maps from $S^6$ to $S^1$ which send $i$ to $1$. Notice that $\JJ_*(S^6) \simeq \Omega^6 \CP^3 \simeq \Omega^6 S^7$. The fiber of $\varpi$ is the space of sections of a trivial $S^1$-bundle over $S^6$, and we have used the canonical trivialization
$$
(p, e^{i\theta}) \mapsto (p, \cos \theta + p\sin \theta) \in 
S^6 \times S^7
$$
to identify this space of sections with $\Map(S^6,S^1)$.


\begin{lemma}
\label{pi1}
Consider the map $\iota$ from diagram \eqref{maind}. The induced map of infinite cyclic groups 
$$\pi_1\Map(S^6,S^1) \xrightarrow{\iota_*} \pi_1\Map(S^6,S^7)$$ 
has as its image the subgroup of index $2$.
\end{lemma}
\begin{proof}
That both groups in the statement are infinite cyclic follows from the  long exact homotopy sequences of the first two columns of 
diagram \eqref{maind}.  
A generator of $\pi_1 \Map(S^6,S^1)$ is given by 
$$ 
e^{i\theta} \mapsto ( p \mapsto e^{i\theta} )
$$
and its image in $\pi_1 \Map(S^6, S^7)$ is given by 
\begin{equation}
\label{loop}
e^{i\theta} \mapsto \left( p \mapsto (\cos \theta + p \sin \theta)  \right).
\end{equation}

By adjunction, there is a canonical isomorphism 
\begin{equation}
\label{identif0}
\pi_1 \Map(S^6,S^7) \simeq \left[ \left( [0,2\pi] \times S^6,\{0,2\pi\} \times S^6 \right) , \left( S^7,1 \right) \right] 
\end{equation} 
of $\pi_1 \Map(S^6, S^7)$ with the group of relative homotopy classes from the cylinder relative to its boundary to the $H$-space $S^7$ (relative
to its basepoint $1$). Precomposition with the quotient map $ [0,2\pi] \times S^6 \xrightarrow{q} S^7$ given by 
$$(\theta, p) \mapsto \cos \frac \theta 2 + p \sin \frac \theta 2 $$
(which collapses the ends of the cylinder to $\pm 1$) gives a 
bijection
\begin{equation}
\label{identif}
[ ([0,2\pi] \times S^6,\{0,2\pi\} \times S^6) , (S^7,1) ]  
\xleftarrow{q^*} [(S^7, \{\pm 1\}), (S^7,1)] \cong 
[(S^7,1), (S^7,1)] = [S^7, S^7]_* 
\end{equation}
where the last isomorphism is a consequence of the identifications
$[(S^7, \{\pm 1\}), (S^7,1)] \cong [S^7\vee S^1, S^7]_* \cong [S^7, S^7]_*$. 

Under the identifications \eqref{identif0} and \eqref{identif}, the loop \eqref{loop} in $\Map(S^6,S^7)$ becomes the squaring map on $S^7$,
$$
\cos \frac \theta 2 + p \sin \frac \theta 2 \mapsto
\cos \theta + p \sin\theta .
$$
Therefore, a generator of $\pi_1 \Map(S^6, S^1)$ is sent to twice a generator of $\pi_1 \Map(S^6, S^7) \cong [S^7, S^7]$, which completes the proof.
\end{proof}

\begin{remark} We can also compute $\pi_1\JJ(S^6)$, and hence obtain an alternative proof of the above lemma, using the following result of Crabb and Sutherland \cite[Proposition 2.7 and Theorem 2.12]{CrSu84}:

\textit{Let $X$ be an oriented closed connected $2n$--manifold and $\xi$ a complex rank $n+1$ bundle over $X$. Denote by $N\xi$ the space of sections of the projective bundle $P\xi$ which lift to sections of $\xi$; this is a non-empty connected space. Then $\pi_1(N\xi)$ is a central extension} $$0 \to \Z/c_n(\xi)[X] \to \pi_1(N\xi) \to H^1(X;\Z) \to 0.$$ We note that our $\CP^3$ bundle $\CP^3 \to SO(7)/U(3) \to S^6$ is the projectivization of the rank 4 complex vector bundle of positive pure spinors $\slashed{S}^{+}$, see \cite[Proposition IV.9.8 and Remark IV.9.12]{LM}. Taking $\xi = \slashed{S}^{+}$, since $H^2(S^6;\ZZ) = 0$, every section of $P\xi$ lifts to a section of $\xi$, i.e. $N\xi = \JJ(S^6)$, and so by the above result and knowledge of the cohomology of $SO(7)/U(3)$ (see e.g. \cite[Proposition 3.2]{P91}) we conclude from the projective bundle formula that $\pi_1 \JJ(S^6) = \Z/c_3(\slashed{S}^{+})[S^6] = \Z/2$. Lemma \ref{pi1} follows from here.
\end{remark}

\begin{cor}
The homotopy groups of the space $\JJ(S^6)$ are given by
$$
\pi_k \JJ(S^6) = \begin{cases}
\Z/2 & \text{ if } k=1, \\
\pi_k(S^7) \oplus \pi_{k+6}(S^7) & \text{ otherwise.}
\end{cases}
$$
\end{cor}

\begin{proof}
Since $S^7$ is an $H$-space we have $\Map(S^6,S^7) \simeq S^7 \times \Map_*(S^6,S^7) = S^7 \times \Omega^6 S^7$. 
The result follows from the long exact sequence in the middle row of \eqref{maind} and Lemma \ref{pi1}.
\end{proof}

Since the map $\Map(S^6, S^7) \xrightarrow{\varpi} \JJ(S^6)$ in diagram \eqref{maind} factors through 
$\Map(S^6, \RP^7)$ (as $-1$ is in the center of $\OO$), the constant sections of $S^6 \times S^7 
\xrightarrow{\pi_1} S^6$ give rise to an inclusion 
\begin{equation}
\label{incrp7}
\RP^7 \stackrel{j}{\hookrightarrow} \JJ(S^6).
\end{equation}
The image of this map is what Battaglia \cite{Ba} calls the space of linear almost complex structures 
on $S^6$ and, as we explain in Proposition \ref{identRP7} below, equals the subspace $SO(7)/G_2$ mentioned in 
the introduction.

\begin{thm}
\label{main}
The inclusion $j$ in \eqref{incrp7} sits in a homotopy fiber sequence 
$$
\Omega^7_0S^7   \to \RP^7 \xrightarrow{j} \JJ(S^6),
$$ 
where $\Omega^7_0S^7$ denotes a path component of the sevenfold based loop space of $S^7$.
\end{thm}
\begin{proof}
Proposition \ref{csr6} (1) identifies the map $\RP^7 \xrightarrow{\ev_i \circ j} \CP^3$ with the Hopf
fibration $\RP^1 \to \RP^7 \to \CP^3$. 

The map of fibrations 
$$
\xymatrix{
\RP^7\ar[r] \ar[d]^{\ev_i \circ j} \ar[r]^j &  \JJ(S^6) \ar[d]^{\ev_i} \\
\CP^3 \ar[r]^= & \CP^3
}
$$
can be extended (see \cite[Section 3.2]{N10}) to a $3\times 3$ diagram of homotopy fiber sequences
\begin{equation}
\label{main3x3}
\xymatrix{
G \ar[r] \ar[d] & \RP^1 \ar[r] \ar[d] & \JJ_*(S^6) \ar[d] \\
F \ar[r] \ar[d] & \RP^7\ar[r]^j \ar[d]^{\ev_i \circ j} &  \JJ(S^6) \ar[d]^{\ev_i} \\
\ast \ar[r] & \CP^3 \ar[r]^= & \CP^3
}   
\end{equation}
We will first prove that it suffices to show that $\RP^1 \to \JJ_*(S^6)$ sends a generator of 
$\pi_1 \RP^1$ to a generator of $\pi_1 \JJ_*(S^6) \cong \pi_1\Omega^6S^7 \cong \Z$. Indeed,
if $g\colon S^1 \to \Omega^6 S^7$ generates $\pi_1$ we may extend the homotopy commutative square
$$
\begin{tikzcd}
 S^1 \arrow[r,"g"] \arrow[d,"="] & \Omega^6 S^7 \arrow[d] \\
S^1 \arrow[r,"="] & S^1
\end{tikzcd}
$$
(where the right vertical arrow classifies the generator of $H^1(\Omega^6 S^7)=\Z$) to a $3\times 3$ square of homotopy
fiber sequences. It is then clear that the space $G$ in \eqref{main3x3} is homotopy equivalent to  $\Omega^7_0 S^7$, from which it follows that the same is true of $F$.

The image of a generator of 
$\pi_1(\RP^1)$ under the map $\RP^1 \to \JJ_*(S^6)$ 
is represented by the loop
$\gamma \colon [0,1] \to \JJ_*(S^6)$ given by
$$
t \mapsto ( p \mapsto p^{R_{\cos \pi t -i\sin \pi t}} ).
$$
A lift $\tilde \gamma$ of this loop to $\Omega^6 S^7$ along
the map $\varpi'$ in \eqref{maind} is given by the formula
$$
t \mapsto (
p \mapsto (\cos \pi t - p \sin \pi t) (\cos \pi t + i \sin \pi t) ).
$$
The image of $\tilde \gamma$ under the canonical identifications $\pi_1\Omega^6 S^7=\pi_0\Omega^7 S^7=[S^7, S^7]$
is the product of the homotopy classes of the maps 
$$ \cos \pi t + p \sin \pi t \mapsto \cos \pi t - p\sin \pi t \quad \text{ and } \quad
\cos \pi t + p \sin \pi t \mapsto \cos \pi t + i\sin \pi t
$$
using the $H$-space structure on $S^7$.  The first map has degree $-1$, while the second one has degree 0. This completes the proof. 
\end{proof}

\begin{cor} The inclusion $\RP^7 \stackrel{j}{\hookrightarrow} \JJ(S^6)$ is an isomorphism on fundamental groups and rational homotopy groups. \end{cor}

\section{Two $\RP^7$s mapping to $\JJ(S^6)$}
\label{rp7}

We will now explain why the subspace $j(\RP^7) \subset \JJ(S^6)$ in Theorem \ref{main}, coming 
from the constant sections of the fibration $S^6 \times S^7 \to S^6$ agrees with the 
orbit $SO(7)/G_2$ of the natural $SO(7)$-action on the canonical element of $\JJ(S^6)$.
This is also explained in \cite[Lemma 2.9, Corollary 2.11, and Proposition 3.5]{Ba} but we include an argument for completeness.

Recall from \cite[Section 8.2]{CS} that $\Spin(8)$ is the group of orthogonal isotopies of the octonions
$$
\Spin(8) = \{ (\phi,\psi,\lambda) \in SO(8)^3 \colon \phi(x)\psi(y)=\lambda(xy) \text{ for all } x,y \in \OO \}.
$$
Each coordinate in a triple $(\phi,\psi,\lambda)$ determines the other two up to a global sign \cite[Theorem 8.3]{CS}. 
The projections $\Spin(8) \to SO(8)$ onto the first two coordinates are the spin representations of $\Spin(8)$ 
while the projection onto the third coordinate is the natural projection to $SO(8)$.

Given $\lambda \in SO(8)$, it is proved in \cite[Theorem 8.11]{CS} that there are
unit octonions $a_\lambda,b_\lambda \in S^7$, which are unique up to sign, called \textit{companions} of $\lambda$, such that the elements $(\phi,\psi,\lambda) \in \Spin(8)$
satisfy
$$
\phi(x) = \lambda(x)a_\lambda, \quad \psi(y) = b_\lambda \lambda(y) \quad \text{ for all } x,y \in \OO.
$$
Note that the classes of the companions in $\RP^7$ can be regarded as the ``difference'' between the $SO(8)$ and the spin representations of $\Spin(8)$ and therefore vary smoothly with $\lambda$. 

Let $SO(7)= \{\lambda \in SO(8) \colon \lambda(1)=1\}$. Then for $(\phi,\psi,\lambda) \in \Spin(8)$ with $\lambda \in SO(7)$ 
we have  $\phi(x) \psi(\overline x) = 1$ for all $x \in \OO$, and hence
$$
\psi(x)= \overline{ \phi(\overline x)}, \quad b_\lambda= \overline a_\lambda .
$$

\begin{prop}
\label{identRP7}
Let $J^{\textrm{cn}}$ denote the canonical almost complex structure from the introduction, and let $j$ denote
the inclusion in \eqref{incrp7}. Then $SO(7)\cdot J^{\textrm{cn}} = j(\RP^7)$.
\end{prop}
\begin{proof}
Given $\lambda \in SO(7)$, $p \in S^6$, $v\in \OO$ and writing $a_\lambda$ for a companion of $\lambda$, we have 
\begin{eqnarray*}
\lambda(\lambda^{-1}(p)\lambda^{-1}(v)) & = & (\lambda(\lambda^{-1}(p)) a_{\lambda} ) (\overline{ a_{\lambda}}\lambda(\lambda^{-1}(v))) \\
& = & (p a_{\lambda} )(\overline{ a_{\lambda} } v)  \\
& = & (p a_{\lambda}) ( \overline{a_{\lambda}} (v a_{\lambda}) \overline{a_{\lambda}} ) \\
& = & (p(v a_{\lambda} ) )\overline{a_{\lambda}}
\end{eqnarray*}
where the last two equalities follow from the Moufang laws \cite[Section 7.4]{CS}.

Hence $\lambda \cdot J^{\textrm{cn}} = j([a_{\lambda}])$ and we have $SO(7)\cdot J^{\textrm{cn}} \subset j(\RP^7)$. The subgroup
$\{ (\phi, \psi, \lambda) \in \Spin(8) \colon \lambda \in SO(7)\}$ is $\Spin(7)$ and the projection 
$$ (\phi,\psi,\lambda) \mapsto \phi $$
is the restriction of a spin representation of $\Spin(8)$ to $\Spin(7)$.
This is the spin representation of 
$\Spin(7)$. Evaluating at $1$ we see that 
$$ \phi(1)= \pm a_\lambda. $$
Since the spin representation of $\Spin(7)$ acts transitively on $S^7$ 
the proof is complete.
\end{proof}

We have been considering the inclusion $SO(7)/G_2 \cong \RP^7 \hookrightarrow \JJ(S^6)$. 
We can also consider the map $\RP^7 \to \JJ(S^6)$ given by the composite
\begin{equation}
\label{other}
S^7 / \mathbb{Z}_2 = \RP^7 \xrightarrow{c} SO(7)  \to SO(7)/G_2  \hookrightarrow \JJ(S^6)
\end{equation}
where the map $c$ takes (the equivalence class of) a unit octonion to its action on the imaginary octonions by conjugation. 

\begin{prop} The composite $\RP^7 \xrightarrow{c} SO(7) \to SO(7)/G_2$ is a degree three self-map of $\RP^7$. Hence the composite $\RP^7 \hookrightarrow \JJ(S^6)$ in \eqref{other} also induces an isomorphism on fundamental groups and on rational homotopy groups. \end{prop}

\begin{proof} Given a unit octonion $x$ and $p \in S^6$, let $A$ be a quaternion algebra containing $p,x$ with standard basis 
$\langle 1, p, \jj, p\jj \rangle$ and decompose $\R^8 \cong \OO$ as $A \oplus IA$. We will compute the action of 
$c(x)$ on the canonical  complex structure $J^{\textrm{cn}}$ in these coordinates. 
$$
(c(x) \cdot J^{\textrm{cn}})_p(v) = \overline x((xp\overline x) (xv\overline x)) x
$$
If $v \in A$, then associativity implies that $(c(x)\cdot J^{\textrm{cn}})_p(v)=pv$. Suppose now that $v=Iw$ with $w\in A$. Then, using $\eqref{CD}$ we obtain
\begin{eqnarray*}
(c(x) \cdot J^{\textrm{cn}})_p(v) & =  & x((\overline x p x) (\overline x(Iw) x)) \overline x \\
& = &   x((\overline xp x) ((I( x w)) x)) \overline x  \\
& = &   x((\overline xp x) (I( x^2w))) \overline x  \\
& = &   x( I((\overline x \, \overline p x)( x^2w)))\overline x  \\
& = &  x ( I (\overline x \, \overline p x^3w)) \overline x \\
& = & I ((\overline x^3 \overline p  x^3) w).
\end{eqnarray*}
We conclude that, with the notation of Proposition \ref{csr6}, we have 
$$ 
(c(x) \cdot J^{\textrm{cn}})_p = p^{R_{x^3}}.
$$
Thus the image of the map \eqref{other} equals $j(\RP^7)$ for the inclusion $j$ in \eqref{incrp7},  but \eqref{other} is generically a 3-to-1 map. Given unit octonions
$x,y$, the elements $[x],[y] \in \RP^7$ have the same image under
\eqref{other} if and only if $x^6=y^6$. \end{proof}

\section{The space of almost complex structures with $c_1=0$ on a six manifold}

We now consider the more general setting of the space of almost complex structures on a six--manifold homotopic to a given $J$ with $c_1(J) = 0$. As before, all almost complex structures will be assumed to be orthogonal with respect to some fixed background metric and compatible with a fixed orientation.


\begin{prop}
\label{gen}
Let $M$ be a six--manifold with an almost complex structure $J$ so that $c_1(J)=0$. Let $f\colon M \to BSU(3)$ be a 
lift\footnote{ Note that this lift is unique up to homotopy through sections, as the homotopy fiber $S^1 \to BSU(3)$ of the map $BSU(3) \to BU(3)$ is null.} of the classifying map of $TM$ determined by $J$. Let $\JJ(M)$ denote the path component of $J$ in the space of
almost complex structures on $M$ and 
$$
S^7 \to E \xrightarrow{q} M
$$
be the $S^7$-bundle on $M$ classified by the composite
$$ 
M \xrightarrow{f} BSU(3) \to BG_2
$$
(where we consider the standard action of $G_2$ on the unit octonions). 
Then there is a fiber sequence 
$$
\Map(M,S^1) \to \Gamma(q) \to \JJ(M).
$$

where $\Gamma(q)$ denotes the space of sections of the bundle $q$. 
\end{prop}
\begin{proof}
The map $M \to BG_2$ gives $TM \oplus \R^2$  the structure of a bundle 
of normed division algebras on $M$. Writing $1,i$ for the
sections corresponding to the last two coordinates, we have that $1$ is the unit on each fiber while $i$ acts as 
multiplication by $J$ on $TM$.

Proposition \ref{csr6} then expresses the bundle of orthogonal almost complex structures on $M$ (compatible 
with the orientation) as the quotient of $E$ by the $S^1$ action given by fiberwise multiplication by $\cos \theta + i \sin 
\theta$. The statement follows. 
\end{proof}

\begin{remark}
Note that $\Gamma(q)$ is an $H$-space. If $M$ is simply connected then $\Map(M,S^1) \simeq S^1$ and so Proposition 
\ref{gen} shows that each component of the space of almost complex structures on $M$ with $c_1=0$ 
is homotopy equivalent to the quotient of an $H$-space by an $S^1$-action. The space $\Gamma(q)$ is the space of sections of the sphere bundle of $TM \oplus \R^2$, and is therefore independent of the choice of almost complex structure. On the other hand, its $H$-space structure (and the corresponding $S^1$-action) will depend on the path component of the almost complex structure.
\end{remark}

If in addition $c_2(J)=0$, then the bundle $q$ in Proposition \ref{gen} 
is trivial as the next lemma shows. In that case we have in the previous 
statement $\Gamma(q) = \Map(M,S^7)$.

\begin{lemma}
\label{c2}
Let $X$ be a six--dimensional cell complex. The composite 
$$
X \xrightarrow{f} BSU(3) \to BG_2
$$
is null if and only if $c_2(f)=0$. 
\end{lemma}
\begin{proof}
Recall that we have a fiber sequence $S^6 \to BSU(3) \to BG_2$ (cf. \eqref{prin}).
Consider the diagram
$$
\begin{tikzcd}
 S^6 \arrow[d] \arrow[r] & K(\Z,6) \arrow[r] & P_6BSU(3) \arrow[d] \\
BSU(3)  \arrow[rr,"c_2"]  \arrow[rru,"p_6"] & & K(\Z,4)  
\end{tikzcd}
$$
where the map $P_6BSU(3) \to K(\Z, 4)$ is the bottom of the Postnikov tower of $BSU(3)$, with homotopy fiber $K(\Z, 6)$. 
If $c_2=0$ then the map $p_6\circ f$ factors through $K(\Z,6)$ and hence through $S^6$.
Since a map from a six complex into $P_6 BSU(3)$ which lifts to $BSU(3)$ does so uniquely, we
see that $f$ lifts to $S^6$. Conversely, if $f$ factors through the homotopy fiber $S^6$ of $BSU(3) \to BG_2$, then clearly $c_2=0$.
\end{proof}

\begin{example}
As an illustration we can compute the homotopy groups of $\mathcal{J}(X_g)$, where $X_g$ denotes the connected sum of $g$ copies of $S^3 \times S^3$, in terms of the homotopy groups of $S^7$. Note that the space of orientation-compatible almost complex structures on $X_g$ is non-empty and connected, and since $c_1 = c_2 = 0$, by Lemma \ref{c2} the tangent bundle map $X_g \to BU(3)$ factors through $S^6$. Since $\chi(X_g) = 2-2g$, this implies the map $X_g \to S^6$ is of degree $1-g$. We thus have the pullback diagram
\begin{equation}
\label{Xg2}
\begin{tikzcd}
 Z(X_g) \arrow[d] \arrow[r] & SO(8)/U(4) \arrow[d]  \\
               X_g \arrow[r, "1-g"]       & S^6      
\end{tikzcd}
\end{equation}
where $Z(X_g) \rightarrow X_g$ denotes the twistor space of $X_g$,
and a corresponding map of fiber sequences
\begin{equation}
\label{mapfibs}
\begin{tikzcd}
\Map(S^6,S^1) \arrow[d, "r"] \arrow[r] & \Map(S^6,S^7) \arrow[d, "s"] \arrow[r] &
\JJ(S^6) \arrow[d,"t"] \\
\Map(X_g,S^1)\arrow[r] & \Map(X_g,S^7) \arrow[r] & \JJ(X_g)\\
\end{tikzcd}
\end{equation}

Here $r,s,t$ denote the natural maps induced by the degree $1-g$ map $X_g \to S^6$.

For a closed oriented 6-manifold $W$, we have 
$\pi_1 \Map(W,S^7) \cong \pi_1 \Map_*(W,S^7) \cong \pi_0 \Map_*(\Sigma W,S^7) \cong \Z$, with the last isomorphism sending
$f \in [\Sigma W, S^7]_*$ to its degree $n$, i.e. the unique integer $n$ such that $f_*(\Sigma[W])=n[S^7] 
\in H_7(S^7)$.
It follows that the map $s$ in the middle column of \eqref{mapfibs} induces multiplication by 
$1-g$ on fundamental groups. Lemma \ref{pi1} together with the long exact sequence of the bottom row of \eqref{mapfibs} then implies that 
$$ 
\pi_1 \JJ(X_g) = \Z/(2-2g).
$$
There is a cofiber sequence 
$$ S^5 \xrightarrow{w} \vee_{i=1}^{2g} S^3 \to X_g $$
where the attaching map $w$ of the top cell is a sum of Whitehead products of inclusions of $S^3$ in the 
wedge
$$ w= [\iota_1,\iota_2] + \ldots + [\iota_{2g-1}, \iota_{2g}] $$
As $\Sigma w$ is null (see for instance
\cite[Theorem X.8.20]{Wh}), we have 
$$
\Sigma X_g \simeq \left( \vee_{j=1}^{2g} S^4 \right) \vee S^7. 
$$
As $S^7$ is an $H$-space, $\Map(X_g,S^7)\simeq S^7 \times \Map_*(X_g,S^7)$ and hence for all $i\geq 1$ we have 
$$
\pi_i \Map(X_g,S^7) \cong \pi_i(S^7) \oplus \pi_{i-1} \Map_*(\Sigma X_g, S^7) 
\cong \pi_i(S^7) \oplus \bigoplus_{j=1}^{2g} \pi_{i+3} S^7 \oplus \pi_{i+6} S^7.
$$
The homotopy long exact sequence of the bottom row of \eqref{mapfibs} then gives us 
$$\pi_2\JJ(X_g) = \begin{cases}
\Z \oplus \Z/2 & \text{ if } g=1, \\
\Z/2 & \text{ otherwise, }
\end{cases}
$$
while the remaining higher homotopy groups agree with those of $\Map(X_g,S^7)$.
\end{example}

\end{document}